\documentclass{article}
\usepackage[all]{xy}
\usepackage{amsmath,amssymb,bm}
\usepackage{amsthm}
\usepackage{graphicx}
\usepackage{indentfirst} 
\usepackage{hyperref}
\usepackage{enumerate}
\usepackage{tikz-cd}
\newtheorem{theorem}{Theorem}[section]
\newtheorem{lemma}[theorem]{Lemma}
\newtheorem{question}[theorem]{Question}

\newtheorem{proposition}[theorem]{Proposition}

\theoremstyle{definition}
\newtheorem{definition}[theorem]{Definition}
\newtheorem{remark}[theorem]{Remark}
\newtheorem*{acknowledgement}{Acknowledgements}

\title{Turaev--Viro invariants and profinite completions of surface bundles}
\author{Qirong Yang}
\date{}
\begin{document}
\maketitle

{\fontfamily{ptm}\selectfont
\begin{abstract}
We prove that the Turaev--Viro invariants of the two surface bundles over the circle coincide for every spherical fusion category if the surface group is procongruently conjugacy separable and there exists a regular profinite isomorphism between the fundamental groups. 
\end{abstract}

\section{Introduction}
The \emph{profinite completion} of group $G$ is the inverse limit of the collection of finite quotients of $G$, denoted by $\widehat{G}$. If a group $G$ is determined by its profinite completion, it is called \emph{profinitely rigid}. In 2018, Reid \cite{Reid2018} gave an ICM talk about main problems and recent progress related to profinite rigidity of 3-manifolds. An interesting topic is to determine the relationship between the profinite completion and other topological invariants. For example, Ueki \cite{Ueki2017} proved the Alexander polynomials are determined by the profinite completions of knot groups. Xu \cite{xuregularity} proved that the A-polynomials of prime knots are also determined by the profinite completions of knot groups. 

Note that the untwisted Dijkgraaf--Witten invariants are determined by the profinite completions of fundamental groups for 3-manifolds and the converse is proved by Funar \cite{Funar2013}. This leads to our first result as follows.

\begin{theorem}\label{main DW}
Let $M_1,M_2$ be closed connected 3-manifolds with finitely generated fundamental groups. They have the same untwisted Dijkgraaf--Witten invariants for every finite group if and only if their fundamental groups are profinitely isomorphic, i.e. $\widehat{\pi_1(M_1)}\cong \widehat{\pi_1(M_2)}$.
\end{theorem}

Dijkgraaf-Witten invariants can be viewed as the special case of Turaev--Viro invariants. This rest of the paper focuses on the Turaev-Viro theory.

It is known that Turaev-Viro invariants in general are not determined by the profinite completions of fundamental groups. It is evident from the explicit values on all lens spaces in \cite{Sokolov1997}. Liu \cite{Liu2023} proved that there exists nontrivial Hempel pairs can be distinguished by Turaev-Viro invariants. However, his result shows there also exist some nontrivial Hempel pairs cannot be distinguished by Turaev-Viro invariants. Meanwhile, Funar \cite{Funar2013} proved that there exist non-homeomorphic but profinitely isomorphic torus bundles over the circle whose monodromies are Anosov.

In this paper, we establish a sufficient condition under which profinitely isomorphic surface bundles over the circle have the same Turaev-Viro invariants for every spherical fusion category. We describe our condition as follows. By Dehn-Nielsen-Baer theorem, we know that there is a natural isomorphism $\operatorname{Mod}(\Sigma)\stackrel{\cong}{\longrightarrow} \mathrm{Out}(\pi_1(\Sigma))$, where $\mathrm{Mod}(\Sigma)$ is the mapping class group of $\Sigma$ and $\mathrm{Out}(\pi_1(\Sigma))$ denotes the outer automorphism group of $\pi_1(\Sigma)$. Let $G$ be a finitely generated group. It is known that there is a canonical homomorphism $\mathrm{Out}(G)\rightarrow \mathrm{Out}(\widehat{G})$ and an inclusion $\mathrm{Out}(G)\rightarrow \widehat{\mathrm{Out}(G)}$. Then there is a homomorphism $\phi$ between two profinite groups :
\begin{displaymath}
\xymatrix{
 \operatorname{Out}(G)\ar[d]\ar[dr] &\\
\widehat{\operatorname{Out}(G)}\ar@{-->}[r]^{\phi} &\mathrm{Out}(\widehat{G})}
\end{displaymath}

\begin{definition}
Let $G$ be a finitely generated residually finite group. It is said to be \emph{procongruently conjugacy separable} if the following property holds: if $f,g\in \mathrm{Out}(G)$ are not conjuagte in $\widehat{\mathrm{Out}(G)}$, then they do not induce a conjugate pair of outer automorphisms in $\mathrm{Out}(\widehat{G})$. 
\end{definition}

Let $G,H$ be finitely generated residually finite groups. An isomorphism $\widehat{G}\rightarrow \widehat{H}$ is said to be \emph{regular} if its abelianization $\widehat{G^{ab}}\rightarrow \widehat{H^{ab}}$ is induced by an isomorphism $G^{ab}\rightarrow H^{ab}$, where $^{ab}$ denotes the abelianization. Now we state our main result:

\begin{theorem}\label{main thm}
Let $\Sigma$ be a closed, orientable, connected surface and let $f, g \in \operatorname{Mod}(\Sigma)$. If $\pi_1(\Sigma)$ is procongruently conjugacy separable and there exists a regular isomorphism $\Phi: \widehat{\pi_1(M_f)} \to \widehat{\pi_1(M_g)}$, then mapping tori $M_f$ and $M_g$ have the same Turaev-Viro invariants for every spherical fusion category.
\end{theorem}

We discuss the two hypothesis of Theorem \ref{main thm}. The conjecture that the regularity of $\Phi$ holds for all finite-volume hyperbolic 3-manifolds remains open in the closed case. Following Xu's proof of regularity for cusped manifolds \cite{xuregularity}, we exhibit a class of closed 3-manifolds  satisfying the regularity condition. 

It is known that there is a natrual epimorphism $$\varphi:\widehat{\mathrm{SL}(n,\mathbb{Z})}\rightarrow \mathrm{SL}(n,\widehat{\mathbb{Z}}).$$ 
Bass, Lazard and Serre \cite{bams} and Mennicke \cite{M} proved $\mathrm{ker}(\varphi)=1$, when $n\geq 3$. 
Thus, $\widehat{\mathrm{SL}(n,\mathbb{Z})}\cong \mathrm{SL}(n,\widehat{\mathbb{Z}})$. Since $\mathrm{Out}(\pi_1(\Sigma))\rightarrow \mathrm{Out}(\widehat{\pi_1(\Sigma)})$ is injective \cite[Corollary 3.6]{BF}, $\phi:\widehat{\mathrm{Out}(\pi_1(\Sigma))}\rightarrow \mathrm{Out}(\widehat{\pi_1(\Sigma)})$ is also an injection. 
However, the following question is widely open,
\begin{question}
Let $G$ be a surface group. Does there exist outer automorphisms $f,g$ of $G$ which are conjugate in $\mathrm{Out}(\widehat{G})$ but not conjugate in $\widehat{\mathrm{Out}(G)}$?
\end{question}
We provide two examples of manifolds satisfying procongruent conjugacy separability. 
\begin{enumerate}[(1)]
\setlength{\itemsep}{0pt}
\item When the surface is an one-punctured $T_{1,1}$ or a closed torus $T^2$, it is known that $\pi_1(T_{1,1})$ and $\pi_1(T^2)$ satisfy the procongruently conjugacy separability since $\mathrm{Mod}(\Sigma)$ is isomorphic to $\mathrm{SL}(2,\mathbb{Z})$. 
\item Let $\Sigma$ be a closed orientable surface. Wilkes \cite{Wilkes2018} proved that if $M_f$ and $M_g$ are closed fibred graph manifolds and $f,g$ are not conjugate in $\mathrm{Out}(\pi_1(\Sigma))$, then $f$ is not conjugate to $g^k$ in $\mathrm{Out}(\widehat{\pi_1(\Sigma)})$ for any $k\in \widehat{\mathbb{Z}}^{\times}$. 
\end{enumerate}

We now outline the proof of Theorem \ref{main thm}. First of all, a regular isomorphism induces the identity map on the fibres. It then follows from Propositon 3.7 of \cite{YiLiu2023} and  Section 5 of \cite{Wilkes2018} that the monodromies $f$ and $g$ induce a conjugate pair of outer automorphisms in $\mathrm{Out}(\widehat{\pi_1(\Sigma)})$. By the condition that $\pi_1(\Sigma)$ is procongruently conjugacy separable, $f$ and $g$ are  conjugate in $\widehat{\mathrm{Out}(\pi_1(\Sigma))} \cong \widehat{\mathrm{Mod}(\Sigma)}$. It remains to prove the following theorem. Note that we prove this result only using the axioms of the TQFT.

\begin{theorem}\label{main TV}
Let $\Sigma$ be a closed orientable connected surface. If the mapping classes $f,g\in \operatorname{Mod}(\Sigma)$ are conjugate in $\widehat{\operatorname{Mod}(\Sigma)}$, the mapping tori $M_f,M_g$ have the same Turaev–Viro invariants for any spherical fusion category.
\end{theorem}

In Section 2, we review the basics of the Turaev--Viro invariants. In Section 3, we give the proof of Theorem \ref{main DW}. In Section 4, we prove that in our hypothesis the monodromies are conjugate in the profinite completion of mapping class group. In Section 5, we prove Theorem \ref{main TV}. In Section 6, we discuss the conditions of Theorem \ref{main thm}.

\begin{acknowledgement}
The author is grateful to Yi Liu for valuable discussions and for the guidance. The author is also grateful to Yifan Jin for helping with the proof.
\end{acknowledgement}

\section{Preliminaries}
The Turaev-Viro invariant for a closed 3-manifold is defined via a state sum model constructed from the data of a spherical fusion category. In this section, we recall the basics about the spherical fusion category and then define the Turaev--Viro invariants. We follow the definitions from \cite{Barrett1993}. 
\subsection{Spherical fusion category}
First we recall the definition of a strict pivotal category. A spherical category is a pivotal category which satisfies an additional condition.
\begin{definition} 
\emph{A category with strict duals} consists of a category $\mathcal{C}$, a functor $\otimes:\mathcal{C}\times \mathcal{C}\rightarrow \mathcal{C}$, an object $e$ and a functor $\widehat{~}:\mathcal{C}\rightarrow\mathcal{C}^{op}$. The conditions are that $(\mathcal{C},\otimes,e)$ is a strict monoidal category, and
\begin{enumerate}[(1)]
\setlength{\itemsep}{0pt}
	\item The functors $\widehat{~}~\widehat{~}$ equals to 1,
	\item The object $\widehat{e}$ equals to $e$,
	\item The functors $\mathcal{C}\times \mathcal{C}\rightarrow \mathcal{C}$ are given by $(a,b)\mapsto(a\otimes b)^{\widehat{~}}$ which equals to $(a,b)\mapsto\widehat{b}\otimes\widehat{a}$.
\end{enumerate}
\end{definition}

\begin{definition}
A \emph{strict pivotal category} is a category with strict duals and a morphism $\epsilon(c):e\rightarrow c\otimes\widehat{c}$ for each $c\in \mathcal{C}$ which satisfies:
\begin{enumerate}[(1)]
\setlength{\itemsep}{0pt}
	\item For all morphisms $f:a\rightarrow b$, the following diagram commutes:
\begin{displaymath}
\xymatrix{
e \ar[r]^{\epsilon(a)}\ar[d]_{\epsilon(b)} & a\otimes\widehat{a}\ar[d]^{f\otimes1}\\
b\otimes\widehat{b}\ar[r]^{1\otimes\widehat{f}}& b\otimes\widehat{a}}
\end{displaymath}
	\item For all objects $a$, the composite $(\epsilon(\widehat{a})\otimes1)\circ(1\otimes\widehat{\epsilon}(a))$ is the identity map of $\widehat{a}$;
	\item For all objects $a,b$, $\epsilon(a\otimes b)=\epsilon(a)\circ(1\otimes(\epsilon\otimes1)):e\mapsto (a\otimes b)\otimes(a\otimes b)^{\widehat{~}}$. The maps $\epsilon$ determine $\widehat{~}$ .
\end{enumerate}
\end{definition}

Let $a$ be an object in a strict pivotal category. Then we define two trace maps $\mathrm{tr}_L,\mathrm{tr}_R:\mathrm{End}(a)\rightarrow\mathrm{End}(e)$ as follows. $\mathrm{tr}_L(f)$ is defined to be the composite 
$$ e \xrightarrow{\epsilon(\widehat{a})} \widehat{a} \otimes a \xrightarrow{1\otimes f} \widehat{a} \otimes a = (\widehat{a} \otimes \widehat{\widehat{a}})^{\widehat{~}} \xrightarrow {\widehat{\epsilon}(\widehat{a})} \widehat{e} = e $$
and $\mathrm{tr}_R(f)$ is defined to be the composite
$$ e \xrightarrow{\epsilon({a})} {a} \otimes \widehat{a} \xrightarrow{f\otimes 1} {a} \otimes \widehat{a} = ({a} \otimes \widehat{{a}})^{\widehat{~}} \xrightarrow {\widehat{\epsilon}({a})} \widehat{e} = e.$$
The two trace maps are symmetric: $\mathrm{tr}_L(fg)=\mathrm{tr}_L(gf)$ and $\mathrm{tr}_R(fg)=\mathrm{tr}_R(gf)$. A pivotal category is called \emph{spherical} if 
$$\mathrm{tr}_L(f)=\mathrm{tr}_R(f)$$
for all objects $a$ and all morphisms $f:a\rightarrow a$. For each object $a$ in a spherical category, its \emph{quantum dimension} is defined to be $\mathrm{dim}_q(a)=\mathrm{tr}_L(1_a)\in \mathrm{End(e)}.$

Let $\mathbb{K}$ be a field. A monoidal $\mathbb{K}$-linear category is a monoidal category $\mathcal{C}$ such that its $\mathrm{Hom}$-sets are $\mathbb{K}$-modules and the composition and monoidal product of morphisms are $\mathbb{K}$-bilinear. An object $a$ is called simple if $\mathrm{End}(a)\cong \mathbb{K}$. An additive category is said to be \emph{semisimple} if every object is a direct sum of finitely many simple objects.

Now a fusion category over $\mathbb{K}$ is a semisimple $\mathbb{K}$-linear category $\mathcal{C}$ with finitely many simple objects. The \emph{dimension} $K$ of a spherical fusion category is defined by $$K=\sum_a \mathrm{dim}_q^2(a).$$

\subsection{Turaev--Viro invariants}
Let $M$ be a closed simplicial manifold, $I$ be the set of simple objects in the category and $K$ be the dimension of the spherical fusion category. Label an element of $I$ to each edge with labelling $l:E\rightarrow I$ where $E$ is the edge set. For each triangle $(a,b,c)$, the state space $H(a,b,c)$ is defined to be $\mathrm{Hom}(b,a\otimes c)$. The state space for the opposite orientation is defined to be the dual space $H^*(a,b,c)$. Let $V(M)$ be the tensor product of the set of triangles of $M$ of the state space of each triangle. By the labelled state sum model(see details in \cite{Barrett1993}), there is a standard linear map $V(M)\rightarrow V(M)$ and the element $Z(M,l)\in \mathbb{K}$ is defined to be the trace of this linear map. We see that $Z(M)=Z(-M)$ which means it does not detect the orientation.

Let $v$ be the number of vertices of $M$. Define the \emph{state sum invariant  (Turaev--Viro invariant)} by a summation over the set of all labellings:
$$Z(M)=K^{-v}\sum_{l:E\rightarrow I}Z(M,l)\prod_{e\in E}\mathrm{dim}_q(l(e)).$$
$Z(M)$ is independent of simplel objects $I$ and independent of triangulation of $M$.  It is known that Turaev--Viro invariant is a $(2+1)$D TQFT derived from the spherical fusion category.


\section{The Profinite Completion and Dijkgraaf--Witten invariants}
In this section, we present the proof of Theorem \ref{main DW}. The crucial step is an algebraic equivalence in Proposition \ref{alg.equiv}. One direction of this equivalence was proved by Funar \cite{Funar2013}, and we complete the proof by establishing the converse.

Given a group $\pi$, we consider the inverse system $\{\pi/N\}$ where $N$ ranges over all finite index normal subgroups of $\pi$. The profinite completion $\widehat{\pi}$ is defined as the inverse limit of this system, i.e. $\widehat{\pi}=\varprojlim{\pi/N}$. The natural homomorphism $\pi\rightarrow \widehat{\pi}$ is injective if and only if $\pi$ is residually finite. From \cite{Hempel87}, we know that fundamental groups of 3-manifolds are residually finite. When $\pi$ is a finitely generated group, a deep result of Nikolov and Segal \cite{Nikolov2006} imples every homomorphism from $\widehat{\pi}$ to a finite group is continuous. Thus, for any finite group $Q$ the homomorphism $\pi\rightarrow \widehat{\pi}$ induces a bijection $\mathrm{Hom}(\widehat{\pi},Q)\rightarrow \mathrm{Hom}(\pi,Q)$. That is to say if $\widehat{\pi_1}\cong\widehat{\pi_2}$, then $|\mathrm{Hom}(\pi_1,Q)|=|\mathrm{Hom}(\pi_2,Q)|$ for any finite group $Q$. Conversely, Funar \cite{Funar2013} shows if two finitely generated groups $\pi_1,\pi_2$ satisfy $|\mathrm{Hom}(\pi_1,Q)|=|\mathrm{Hom}(\pi_2,Q)|$ for any finite group $Q$, then the sets of finite quotients of $\pi_1$ and $\pi_2$ coincide. It is known that two finitely generated groups have the same sets of finite quotients if and only if  their profinite completions are isomorphic. Therefore, we have obtained 
\begin{proposition}\label{alg.equiv}
Let $\pi_1,\pi_2$ be finitely generated groups. The followings are equivalent:   
\begin{enumerate}[(1)]
\setlength{\itemsep}{0pt}
	\item $|\mathrm{Hom}(\pi_1,Q)|=|\mathrm{Hom}(\pi_2,Q)|$ for any finite group $Q$, 
	\item the profinite completions of $\pi_1$ and $\pi_2$ are isomorphic.
\end{enumerate}
\end{proposition}

Recall the definition of Dijkgraaf-Witten invariants from \cite{Dijkgraaf1990} which are topological invariants of 3-manifolds defined via a finite-group $G$, modulated by a cohomology class $\alpha \in H^3(G, U(1))$. They provide a classification of manifolds using finite group data and can be extended to define a topological quantum field theory (TQFT). The explicit formula is defined as
$$ Z_{\alpha}(M,G) = \frac{1}{|G|}\sum_{\varphi\in \operatorname{Hom}(\pi_1(M),G)}\langle\tilde{\varphi}^*(\alpha),[M]\rangle,$$
where $\tilde{\varphi}:M\rightarrow BG$ is induced by $\varphi$, $[M]$ is the fundamental class of $M$ and $\langle\cdot,\cdot\rangle$ is the evaluation pairing.
In the untwisted case ($\alpha$ is trivial), the invariant $Z(M,G)$ for a closed 3-manifold $M$ is given by the following counting formula, 
\[
Z(M,G) = \frac{1}{|G|}|\operatorname{Hom}(\pi_1(M),G)|.
\]
(For details, see \cite[(5.14)]{Freed1991}.) Then Theorem \ref{main DW} is a corollary from the Proposition \ref{alg.equiv}.

\section{Profinite isomorphism}

Let $\Sigma$ be a closed surface. Suppose surface group $\pi_1(\Sigma)$ is procongruently conjugacy separable. In this section, we prove that monodromies $f,g\in \mathrm{Mod}(\Sigma)$ are conjugate in the profinite completion $\widehat{\mathrm{Mod}(\Sigma)}$ when there is a regular isomorphism $\widehat{\pi_1(M_f)}\rightarrow \widehat{\pi_1(M_g)}$.

For an orientation preserving self-diffeomorphism $f:\Sigma\rightarrow \Sigma$, the mapping torus of $f$ is defined as
$$M_f=\frac{\Sigma\times I}{(x,0)\sim (f(x),1)}.$$

There is a canonical short exact sequence of fundamental groups:
$$1 \rightarrow\pi_1(\Sigma)\rightarrow\pi_1(M_f)\stackrel{\phi_f}\longrightarrow\mathbb{Z}\rightarrow 1,$$
induced by the surface bundle $\Sigma \hookrightarrow M_f \rightarrow S^1$.
There is a commutative diagram of group homomorphisms:
\begin{displaymath}
\begin{tikzcd}
1 \ar[r]& \pi_1(\Sigma)\ar[r]\ar[d,"i"] & \pi_1(M_f)\ar[r,"\phi_f"] \ar[d,"i"] & \mathbb{Z}\ar[r]\ar[d,"i"] & 1\\
1 \ar[r]& \widehat{\pi_1(\Sigma)}\ar[r] & \widehat{\pi_1(M)}\ar[r,"\widehat{\phi_f}"] & \widehat{\mathbb Z}\ar[r] & 1
\end{tikzcd}
\end{displaymath}
where both rows are exact and the vertical homomorphisms $i$ are natural inclusions.

\begin{lemma}
If the isomorphism $\Phi:\widehat{\pi_1(M_{f})}\rightarrow \widehat{\pi_1(M_{g})}$ between profinite completions is regular, then we have the following commutative diagram:
\begin{displaymath}
\begin{tikzcd}
\widehat{\pi_1(M_f)}\ar[r,"\widehat{\phi_f}"]\ar[d,"\Phi"]&\widehat{\mathbb{Z}}\ar[d,"\pm1"]\\
\widehat{\pi_1(M_g)}\ar[r,"\widehat{\phi_g}"]&\widehat{\mathbb{Z}}
\end{tikzcd}
\end{displaymath}
\end{lemma}
\begin{proof}
By the surface bundle $\Sigma \hookrightarrow M_f \rightarrow S^1$ and the regularity,  we have the following two commutative diagrams:

\begin{displaymath}
\begin{tikzcd}
 \pi_1(M_f)\ar[r]\ar[d,"\phi"] & H_1(M_f,\mathbb{Z})\ar[r,"\alpha_f"] \ar[d,"\varphi"] & \mathbb{Z}\ar[d,"\mu"] \\
 {\pi_1(M_g)}\ar[r] & H_1(M_g,\mathbb{Z})\ar[r,"{\alpha_g}"] & {\mathbb Z}
\end{tikzcd}
\end{displaymath}

\begin{displaymath}
\begin{tikzcd}
 \widehat{\pi_1(M_f)}\ar[r]\ar[d,"{\Phi}"] & \widehat{H_1(M_f,\mathbb{Z})}\ar[r,"\widehat{\alpha_f}"]\ar[d,"\widehat{\varphi}"] & \widehat{\mathbb{Z}}\ar[d,"\widehat{\mu}"] \\
 \widehat{\pi_1(M_g)}\ar[r] & \widehat{H_1(M_g,\mathbb{Z})}\ar[r,"\widehat{\alpha_g}"] & \widehat{\mathbb Z}
\end{tikzcd}
\end{displaymath}
where the isomorphisms in the second diagram are induced by isomorphisms in the first diagram.
Since $\Phi$ is an isomorphism, we know that $\widehat{\varphi}$ and $\widehat{\mu}$ are also isomorphisms. Since $\widehat{\mu}$ is an automorphism of $\widehat{\mathbb{Z}}$, and$$\widehat{\mathbb{Z}}^{\times}\cong \prod_{p\ prime} {\mathbb{Z}_p^{\times}},$$  $\widehat{\mu}$ must be the scalar multiplication by a unit in $\widehat{\mathbb{Z}}$. By the regularity, $\varphi$ is an isomorphism, so we have $\varphi (ker \alpha_f)=ker \alpha_g$. Consider the semiproduct structure of $H_1(M_g,\mathbb{Z})\cong H_1(\Sigma) \rtimes_{\phi} \mathbb{Z}$. Since $\varphi$ is an isomorphism preserving $\ker \alpha_f = H_1(\Sigma)$ and $\alpha_g \circ \varphi = \mu \circ \alpha_f$, its action on the splitting must take the form
\[
\varphi(v, n) = (\psi(v) + \gamma_n, \mu(n)),
\]
where $\psi \in \mathrm{Aut}(H_1(\Sigma))$, $\gamma \in H_1(\Sigma)$ and $\gamma_n=\gamma+\phi(\mu(1))(\gamma_{n-1})$.  In particular, $\varphi$ sends the generator $t = (0,1)$ to $(\gamma, \mu(1))=\mu(1)t+\gamma$. Since $\varphi$ is an isomorphism, $\mu(1)$ coincides to $\pm 1$, the invertible elements of $\mathbb{Z}$. Thus, we get $\widehat{\mu}\in \{\pm 1\}.$
\end{proof}

\begin{remark}
We know that the definition of Turaev--Viro invariants does not use an orientation of manifolds. Thus, the following computation for the case $\widehat{\mu}=-1$ is the same as $\widehat{\mu}=+1$. Therefore, we just consider $\widehat{\mu}=+1$ in our paper if there is a regular isomorphism of profinite completions.
\end{remark}

\begin{definition}[{\cite{YiLiu2023}}]
Let $\Sigma$ be an orientable connected compact surface and $f,g \in \operatorname{Mod}(\Sigma)$. An isomorphism $\Phi:\widehat{\pi_1(M_{f})}\rightarrow \widehat{\pi_1(M_{g})}$ is called \emph{aligned} if it satisfies the following commutative diagram of group homomorphism:
\begin{displaymath}
\begin{tikzcd}
 \widehat{\pi_1(M_{f})}\ar[r,"\widehat{\phi_f}"]\ar[d,"\Phi"] &\widehat{\mathbb{Z}}\ar[d,"\mathrm{id}"]\\
\widehat{\pi_1(M_{g})}\ar[r,"\widehat{\phi_g}"] &\widehat{\mathbb{Z}}
\end{tikzcd}
\end{displaymath}
\end{definition}

\begin{definition}[{\cite{YiLiu2023}}]\label{procongruently conjugate}
Let $G$ be a finitely generated residually finite group. A pair of outer automorphisms are said to be \emph{procongruently conjugate} if they induce a conjugate pair of outer automorphisms in $\operatorname{Out}(\widehat{G}) $.
\end{definition}

\begin{lemma}[{\cite{YiLiu2023}}, see also in Section 5 of \cite{Wilkes2018}]
In any dialogue setting $(\Sigma,f,g)$, there exists an aligned isomorphism from $\widehat{\pi_1(M_{f})}$ to $\widehat{\pi_1(M_{g})}$ is equivalent to the mapping classes $f,g\in\operatorname{Mod}(\Sigma)$ are procongruently conjugate.
\end{lemma}

Since $G$ is a finitely generated group, $\mathrm{Aut}(\widehat{G})$ is a profinite group. Therefore $\mathrm{Out}(\widehat{G})$ is a profinite group and there exists a homomorphism $\bar{\varphi}$ such that the following universal property is satisfied:
\begin{displaymath}
\xymatrix{
 \widehat{\operatorname{Out}(G)}\ar@{-->}[dr]^{\bar{\varphi}} &\\
\operatorname{Out}(G)\ar[u]^{\iota}\ar[r]^{\varphi} &\mathrm{Out}(\widehat{G})}
\end{displaymath}

\begin{definition}
Let $G$ be a finitely generated residually finite group. If for any $f,g\in \operatorname{Out}(G)$ not conjugate in  $\widehat{\operatorname{Out}(G)}$, their images in $\mathrm{Out}(\widehat{G})$ are not conjugate. The group $G$ is said to be \emph{procongruently conjugacy seperable}.
\end{definition}

Since $\pi_1(\Sigma)$ is profinitely conjugacy seperable, $f,g$ are conjuagte in $\widehat{\operatorname{Mod}(\Sigma)}$. Then combining with Theorem \ref{main TV}, Theorem \ref{main thm} follows.

\section{Proof of Theorem \ref{main TV}}

Turaev--Viro invariants construct a functor $Z^{TV}$ from the cobordism category to the category of vector spaces over the field $\Bbbk$. The functor $Z^{TV}$ assigns a vector space $Z^{TV}(\Sigma)$ to the closed surface $\Sigma$ and an element \(Z^{TV}(M_f)\in Z^{TV}(\partial M_f)=\Bbbk\) to the mapping torus \(M_f\). Turaev and Viro proved that $\mathrm{TV}(M_f)=Z^{TV}(M_f)$ and $Z^{TV}$ satisfies the TQFT axioms \cite{TURAEV1992}. By Atiyah’s TQFT axioms \cite{Atiyah1988}, the functor $Z^{TV}$ naturally induces a representation 
$$\rho^{TV}:\mathrm{Mod}(\Sigma)\rightarrow \mathrm{GL}(Z^{TV}(\Sigma)).$$
Then the following result is important in our proof,

\begin{lemma}[\cite{Atiyah1988}]
For an oriented closed surface $\Sigma$, and any mapping class $[f]\in \mathrm{Mod}(\Sigma)$,
\[
\mathrm{TV}(M_f) = \operatorname{tr}(\rho^{TV}([f])).
\]
\end{lemma}

The topological group $X$ with continuous homomorphisms $\phi_i$ are called \emph{compatible} with system $(X_i,\phi_{ij},I)$ if the following diagram is commutative:
\begin{displaymath}
\begin{tikzcd}
 X\ar[r,"\phi_i"]\ar[dr,"\phi_j"'] &{X_i}\ar[d,"\phi_{ij}"]\\
 &X_j
\end{tikzcd}
\end{displaymath} 

\begin{proposition}[the universal property, {\cite{Ribes2000}}]
The inverse limit system $(X_i,\phi_{ij},I)$ together with compatible continuous homomorphisms
$$\phi_i:{X}\rightarrow X_i$$
satisfies the \emph{universal property}:

\begin{displaymath}
\begin{tikzcd}
 Y\ar[r,"\psi"]\ar[dr,"\psi_i"'] &{X}\ar[d,"\phi_i"]\\
 &X_i
\end{tikzcd}
\end{displaymath} 
whenever $Y$ is a topological group and $\psi_i:Y\rightarrow X_i$ is a set of compactible continuous homomorphisms, then there is a unique continuous homomorphism $\psi:Y\rightarrow {X}$ such that $\phi_i\psi=\psi_i$ for all $i\in I$.
\end{proposition}

\begin{lemma}\label{mod}
 Let $R$ be a finitely generated $\mathbb{Z}$-algebra, $a,\ b\in R$. If $a\equiv b \ (mod \ I)$ for all finite index ideals $I$, then $a=b$ in $R$.
\end{lemma}

To prove this, we need the following two lemmas.

\begin{lemma}[Artin--Rees, {\cite[Corollary 10.10]{Ati}}]\label{Artin-Rees}
    Let $R$ be a Noetherian ring with an ideal $I$, $M$ be a finitely generated $R$-module with submodule $N$. Then there exists a natural number $c\in\mathbb{N}$, such that for any $n\in\mathbb{N}$, we have 
    \[ 
        I^{n+c}M\cap N=I^n(I^cM\cap N).
    \]
\end{lemma}

\begin{lemma}[{\cite[Chapter V, \S 3.4, Corollary 1]{Bou}}]\label{finite}
    Let $R$ be a finitely generated $\mathbb{Z}$-algebra that is a field, then $R$ is finite.
\end{lemma}

\begin{proof}[Proof of Lemma {\ref{mod}}.]
    We need to show that if $a-b\in I$ for any finite index ideal $I$, then $a-b=0$. It suffices to prove that for any nonzero element $x\in R$, there exists a finite index ideal $I$ that does not contain $x$. Using Zorn’s lemma, we can consider the ideal $I$ that is maximal subject to the condition that it does not contain $x$. Our goal is to show that $R/I$ is finite.
    
    Note that every nonzero ideal in $R/I$ contains $x$. So $J/I:=I+(x)/I$ is the unique minimal nonzero ideal of $R/I$. Since $J/I$ only has trivial submodules, it is an irreducible $R$-module. Consider the $R$-module homomorphism 
    $$
    \varphi:R\to J/I,\quad r\mapsto rx.
    $$
    Then $\varphi$ is surjective (otherwise its image is a nontrivial submodule of $J/I$) and its kernel is a maximal ideal $M$ (otherwise $\varphi(N/M)$ is a nontrivial submodule of $J/I$ for some maximal ideal $N$ containing $M$) containing $I$. Hence $\varphi$ induces an isomorphism $R/M\simeq J/I$ and $MJ/I=0$. 

    Since $R$ is a finitely generated $\mathbb{Z}$-algebra, we can apply Lemma \ref{Artin-Rees} to $R$ with maximal ideal $M$, and the finitely generated $R$-module $R/I$ with submodule $J/I$, then there exists a natural number $c\in\mathbb{N}$ such that $((M^{c+1}+I)\cap J)/I=M((M^c+I)\cap J)/I$. Note that $MJ/I=0$ implies that $(M^{c+1}+I)\cap J\subset I$. If $M^{c+1}+I\not\subset I$, we have $x\in M^{c+1}+I$ and therefore $J\subset(M^{c+1}+I)\cap J\subset I$, a contradiction. So $(M/I)^{c+1}=(M^{c+1}+I)/I=0$. This implies that $M/I$ is contained in the nilradical ideal of $R/I$ and therefore the Jacobson radical ideal of $R/I$. Thus, $R/I$ is a local ring.

    Moreover, since $R/M$ is a finitely generated $\mathbb{Z}$-algebra that is a field, Lemma \ref{finite} implies that $R/M$ is a finite field. Therefore, $(M/I)^{c+1}=0$ implies that $R/I$ is finite.
\end{proof}

\begin{proof}[Proof of Theorem \ref{main TV}]

Let $\Sigma$ be a connected closed orientable surface of genus $g$ and let $\Bbbk$ be a field. Let $Z(\Sigma)$ be the assigned vector space over $\Bbbk$. From \cite[section 9.5]{TURAEV1992}, we know that the vector space $Z(\Sigma)$ is $\Bbbk[H_1(\Sigma;\mathbb{Z}_2)]$, freely generated by elements of $H_1(\Sigma;\mathbb{Z}_2)$ over $\Bbbk$. Note that $H_1(\Sigma;\mathbb{Z}_2)=\mathbb{Z}_2^{2g}$ for $\Sigma$ of genus $g$, we have $Z(\Sigma)=\Bbbk^n,\ n=2^{2g}$. Then the linear representation  is denoted as $\rho:\operatorname{Mod}(\Sigma)\rightarrow \mathrm{GL}(n,\Bbbk)$. 

Wajnryb \cite{WAJNRYB1996} proved that the mapping class group of a closed oriented surface is generated by two elements. Suppose $\operatorname{Mod}(\Sigma)$ is generated by $x_1,x_2$ over $\mathbb{Z}$. Let $\bar{R}$ be  the subring of $\Bbbk$ generated by all the entries of the matrices $\rho(x_i)^{\pm 1},\ i=1,2$.

Hence the linear representation is finally denoted as $$\rho:\operatorname{Mod}(\Sigma)\rightarrow \mathrm{GL}(n,\bar{R})$$ for the ring $\bar{R}$ defined above.

Note that $\mathrm{GL}(n,\bar{R}/I)$ are finite groups for every finite index ideal $I$ of $\bar{R}$. Define $\rho_I=p_I\circ \rho$, where $p_I$ is the projection$$\rho_I:\operatorname{Mod}(\Sigma) \stackrel{\rho}{\longrightarrow}\mathrm{GL}(n,\bar{R})\stackrel{p_I}{\longrightarrow} \mathrm{GL}(n,\bar{R}/I).$$
Then for the universal property, $\rho_I$ factors through $\widehat{\operatorname{Mod}(\Sigma)}$, i.e.there is a unique continuous map such that $\rho_I=\bar{\rho}\circ \eta$.
\begin{displaymath}
\begin{tikzcd}
\operatorname{Mod}(\Sigma)\ar[dr,"\eta"']\ar[rr,"\rho_I"] && \mathrm{GL}(n,\bar{R}/I)\\
 &\widehat{\operatorname{Mod}(\Sigma)}\ar[ur,"\bar{\rho}"']&
\end{tikzcd}
\end{displaymath}

Since $f$ and $g\in \operatorname{Mod}(\Sigma)$ are conjugate in the profinite completions $\widehat{\operatorname{Mod}(\Sigma)}$ of mapping class group, there is a diffeomorphism $T\in \operatorname{Mod}(\Sigma)$ such that $\eta([f])=\eta([T])\eta([g])\eta([T])^{-1}$. 

Therefore,
\begin{align*}
\mathrm{tr}(\rho_I[f])&=\mathrm{tr}(\bar{\rho}(\eta([f]))\\
&=\mathrm{tr}(\bar{\rho}\eta([T])\cdot\bar{\rho}\eta([g])\cdot\bar{\rho}\eta([T])^{-1})\\
&=\mathrm{tr}(\rho_I[g])
\end{align*} for all finite index ideal $I$ of $\bar{R}$.

Moreover, $\mathrm{tr}(\rho_I([f]))=\mathrm{tr}(\rho([f]))\ mod\ I$ for every $f\in \operatorname{Mod}(\Sigma)$.
Since the ring $\bar{R}$ is a finitely generated $\mathbb{Z}$-algebra, Lemma \ref{mod} implies the conclusion $\mathrm{TV}(M_f)=\mathrm{TV}(M_g)$. This completes the proof of Theorem \ref{main TV}.

\end{proof}

\section{The regularity}
In this section, we exhibit a class of closed 3-manifolds that satisfies the regularity condition. For finite-volume hyperbolic 3–manifold groups, one might expect that all profinite isomorphisms are regular, but that conjecture remains the closed case hard to approach. Liu \cite{Liu2020} showed that all profinite isomorphisms are $\widehat{\mathbb{Z}}^{\times}$-regular in the finite volume hyperbolic case including closed manifolds. Xu \cite{xuregularity} proved that all profinite isomorphisms are regular in the cusped case. By performing Dehn filling on such cusped manifolds, we construct examples of closed manifolds that also satisfy the regularity condition.

Let $M$ be a finite volume hyperbolic 3-manifold with $n>0$ cusps denoted as $\partial_1 M,\dots,\partial_n M$. For each $\partial_i M\cong T^2$, $c_i$ is a slope on $\partial_i M$ which is allowed to be empty. We obtain a closed manifold $M_c=M_{c_1,\dots,c_n}$ by performing Dehn filling along the slopes $c_i$ on each boundary torus.

\begin{theorem}[{\cite[Theorem A]{xuwitnessed}}]\label{Xu}
Suppose $M,N$ are finite volume cusped hyperbolic 3-manifolds and $f:\widehat{\pi_1(M)}\rightarrow \widehat{\pi_1(N)}$ is an isomorphism. Then there is a homeomorphism $\Psi:\partial M\rightarrow \partial N$ and an isomorphism $\bar{f}:\widehat{\pi_1(M_c)}\rightarrow \widehat{\pi_1(N_{\Psi(c)})}$ such that the following diagram commutes
\begin{displaymath}
\begin{tikzcd}
\widehat{\pi_1(M)}\ar[r,"\widehat{p}"]\ar[d,"f"] &\widehat{\pi_1(M_c)}\ar[d,"\bar{f}"]\\
 \widehat{\pi_1(N)}\ar[r,"\widehat{q}"]&\widehat{\pi_1(N_{\Psi(c)})}
\end{tikzcd}
\end{displaymath}
where $\widehat{p},\widehat{q}$ are induced by inclusions.
\end{theorem}

\begin{theorem}
Suppose $M,N$ are finite volume oriented cusped hyperbolic 3-manifolds. The isomorphism between the profinite completion of fundamental groups of the closed 3-manifolds $\bar{f}:\widehat{\pi_1(M_c)}\rightarrow \widehat{\pi_1(N_{\Psi(c)})}$ induced by isomorphism $f:\widehat{\pi_1(M)}\rightarrow \widehat{\pi_1(N)}$ is regular.
\end{theorem}
\begin{proof}
Let $[M]\in H_3(M,\partial M;\mathbb{Z})$ and $[N]\in H_3(N,\partial N;\mathbb{Z})$ be the fundamental class. The \emph{profinte mapping degree} $deg(f)$ is defined as
\begin{align*}
f_*:\widehat{\mathbb{Z}}\otimes H_3(M,\partial M;\mathbb{Z})&\rightarrow \widehat{\mathbb{Z}}\otimes H_3(N,\partial N;\mathbb{Z})\\
1\otimes [M]&\mapsto deg(f)\otimes [N].
\end{align*}
By Xu's result in \cite{xuregularity}, $deg(\bar{f})=deg(f)=\pm 1$, then $\bar{f}$ is  a regular isomorphism.
\end{proof}

Therefore, two closed surface bundles over the circle obtained by the Dehn filling in Theorem \ref{Xu} satisfying the regularity condition.

\bibliographystyle{plain}
\bibliography{refs}
\end{document}